\documentclass{amsart}
\usepackage{chngcntr}
\usepackage{apptools}
\usepackage{graphicx}
\usepackage{mathrsfs}
\usepackage{enumerate}
\usepackage[latin1]{inputenc} 
\usepackage{color}
\usepackage{amsthm,amssymb,verbatim}
\usepackage{enumerate}

\newcommand{\rr}{{\mathbb R}}
\newcommand{\nn}{{\mathbb N}}
\newcommand{\Sp}{{\mathbb S}}

\newcommand{\eps}{\varepsilon}
\newcommand{\set}[1]{\left\{#1\right\}}
\newcommand{\pare}[1]{\left(#1\right)}
\newcommand{\abs}[1]{\left|#1\right|}
\newcommand{\norm}[1]{\abs{\abs{#1}}}

\newcommand{\A}{{\mathcal A}}

\newcommand{\pI}[1]{\left <{#1}\right >}

\newcommand\restr[2]{{
  \left.\kern-\nulldelimiterspace 
  #1 
  \vphantom{\big|} 
  \right|_{#2} 
  }}

 \newtheorem{theorem}{Theorem}[section]
\newtheorem{lemma}[theorem]{Lemma}
\newtheorem{proposition}[theorem]{Proposition}
\newtheorem{corollary}[theorem]{Corollary}
\newtheorem{definition}[theorem]{Definition}
\newtheorem{notation}[theorem]{Notation}
\newtheorem{remark}[theorem]{Remark}

\newtheorem{example}[theorem]{Example}

\usepackage{hyperref}
\usepackage[alphabetic, msc-links, backrefs]{amsrefs}

\begin{document}

\title[Translating soliton for $Q_k$-flows]{Interior estimates for Translating solitons of the $Q_k$-flows in $\rr^{n+1}$}
\author[J. Torres Santaella]{Jos\'e  Torres Santaella} 
\address{Departamento de Matemáticas\\ Pontificia Universidad Católica de Chile\\ Santiago\\ Chile}
\email{jgtorre1@uc.cl}

\maketitle

\begin{abstract}
 We prove interior gradient estimate and second order estimate for the $Q_k$-flow and $Q_k$-translators in $\rr^{n+1}$. In addition, we show that $Q_k$-translator which are asymptotic to $o(|x|)$ cannot exist. 
\end{abstract}

\section{Introduction}
Geometric evolution problems for hypersurfaces have had a remarkable development over the last decades, since this kind of problems lead to interesting non-linear PDE's that have been used to solve important open questions in mathematics and physics. 
\newline
In this paper, we are interested in a particular class of extrinsic flows, where the speed of the flow is given by a $1$-homogeneous  function in the principal curvatures of the given initial hypersurface. 
\newline
More precisely, given a manifold $M^n$ and an immersion $F_0:M^n\to\rr^{n+1}$, one wants to find a $1$-parameter family of immersions $F:M\times[0,T)\to\rr^{n+1}$ such that 
\begin{align}\label{I.2}
\begin{cases}
    \left(\dfrac{\partial F}{\partial t}\right)^{\perp}&=Q_k(\lambda),\mbox{ in }M\times(0,T),
    \\
    F(\cdot,0)&=F_0(\cdot),
    \end{cases}
\end{align}
where $(\cdot)^{\perp}$ means the orthogonal projection onto the normal bundle of $M$, 
\begin{align*}
Q_k=\dfrac{S_{k+1}}{S_k},
\end{align*}
and $S_k(\lambda)$ denotes the elementary symmetric polynomial in $n$ variables of degree $k$ evaluated at the principal curvatures of $M_t=F(M,t)$, i.e:
\begin{align}\label{I.1}
    S_{k+1}(\lambda)=\sum_{1\leq i_1<i_2<\ldots i_{k+1}\leq n}\lambda_{i_{1}}\ldots\lambda_{i_{k+1}}.
\end{align}
Note that by definition $S_0=1$ and $Q_0=H$ is the scalar mean curvature, which has been widely studied in this century see \cite{MR3312634} for an introduction survey. In addition, these type of functions are the simplest case of geometric quotients of curvature $1$-homogeneous functions.

From the PDE perspective, Equation \ref{I.2} is locally uniformly parabolic when the principal curvature vector $\lambda$ lies in the cone
\begin{align*}
\Gamma_{k+1}=\set{\lambda\in\rr^{n}:S_l(\lambda)>0,\:l=0,\ldots,k+1}.
\end{align*}
In particular, this fact motivate the study of the $Q_k$-flow in the following papers \cite{andrews_2004}, \cite{dieter_2005} and \cite{Choi-Daskalopoulos_2016}, which we describe below. In \cite{andrews_2004},
the author shows classical results about existence, uniqueness and collapsing types theorem for closed strictly convex initial hypersurfaces.
\newline
A year later, in \cite{dieter_2005}, the author uses a weaker condition on the convexity of the initial hypersurface to show similar results as in \cite{andrews_2004}. Namely, one can start the $Q_k$-flow when $\lambda\in\Gamma_k\cap\set{\lambda\geq0}$ instead of being strictly convex. This means that we can permit that $Q_k$ vanish at some points in $M_0$.  
\newline
Finally, in \cite{Choi-Daskalopoulos_2016}, the authors show an existence result of non-compact complete convex solutions in the spirit of \cite{ecker_huisken_1989} for the Mean Curvature Flow.
\newline

In this paper we focus on eternal solutions\footnote{Solution which exists for all time.} of the $Q_k$-flow which evolve by translation, usually known as translating solitons, or $Q_k$-translators for short. These type of solutions are immersion of the form 
\begin{align*}
F(x,t)=F_0(x)+tv,
\end{align*}
where $v\in\Sp^n$ is a fixed direction. Since we are interested in the normal part, $Q_k$-translators can also be seen as hypersurfaces $M_0\subset\rr^{n+1}$ which satisfy,
\begin{align}\label{I.3}
Q_k(\lambda)=\pI{\nu,v}.
\end{align} 
These type of solitons have been studied by many authors in the case $k=0$, we recommend the reader to see the survey \cite{Hoffman_Martin_white_2019} as a guide in this area. 
\newline
The interest on these examples arises mainly due to two reasons: they appear as a model for type $2$ singularities of the Mean Curvature Flow \cite{huisken-sinestrari_1999} and  they are also minimal surfaces in weighted Euclidean spaces \cite{ilmanen_1994}.
\newline

In this paper, we obtain several results for the $Q_k$-flow and $Q_k$-translators in $\rr^{n+1}$. In the remains of this section, we only summarize the results obtained for $Q_k$-translators, since the statement for the $Q_k$-flow are similar. 
\newline
Firstly, by adapting the gradient estimates founded in \cite{sheng_urbas_wang_2004} into our settings, we obtain a local gradient estimate given below.
\begin{theorem}
Let $r>0$ and $u\in\mathcal{C}^3\left(B(0,r)\right)$ such that $\mbox{graph}(u)$ is a $Q_k$-translator. Then it holds 
\begin{align*}
    |Du(0)|\leq\exp\left(\dfrac{CM}{r}+\dfrac{CM^2}{r^2}\right),
\end{align*}
where $M=\sup_{B(0,r)}u$ and $C=C(k,n)$. 
\end{theorem}
Moreover, by assuming $k\geq 1$, any hypersurface $M$ with its principal curvatures lying in $\Gamma_{k+1}$, satisfies  $|\A|^2\leq H^2$ .Therefore, by applying a similar proof of second order estimate given in \cite{ecker_huisken_1991},  we obtain second order estimates for solution in both settings
\begin{theorem}
Let $p\in M$ where $M$ is a $Q_k$-translator with $k\geq 1$, then $|A|^2$ is uniformly bounded in a $B(p,R)$.
\end{theorem}

The organization of this papers goes as follows: In Section \ref{section:Properties} we deal with properties for the $Q_k$ functions that we will use along the paper. In Section \ref{sec3} we proves gradient estimates for graph solutions to the $Q_k$-flow and the $Q_k$-translator equation. We also prove a non-existences theorem for linear growth solutions of the graph $Q_k$-translator equation. In Section \ref{sec4} we prove second order estimate in the above setting.
\vskip 3mm
\noindent {\bf Acknowledgments:} The author would like to thank F. Martín and M. S\'aez for bringing this problem to his attention and for all the support they have provided. Furthermore, he would like to thank Pontificia Universidad Católica de Chile and Universidad de Granada for their facilities, seminars and doctoral schools on Differential Geometry in which he participated.
\section{Properties of $Q_k$ functions}\label{section:Properties}

In this section we list some properties of the $S_k$ and $Q_k$ functions in $\rr^n$.

\begin{definition}
By setting $S_0:=1$, $S_k:=0$ for $k>n$ and $S_{k}$ as in the formula (\ref{I.1}) for $k=1,\ldots,n$. We define the open convex cones
 \begin{align*}
    \Gamma_k:=\set{\lambda\in\rr^n:\:S_i(\lambda)>0,\mbox{ for }i=1,\ldots,k}.
\end{align*}
\end{definition}

\begin{example}
The most common examples of $Q_k(\mu)$ functions corespond to
\begin{align*}
Q_0(\lambda)=\lambda_1+\ldots+\lambda_n\mbox{  and }Q_{n-1}(\lambda)=\left(\dfrac{1}{\lambda_1}+\ldots+\dfrac{1}{\lambda_n}\right)^{-1},
\end{align*}
when we consider these expressions in (\ref{I.2}) we have the Mean and the Inverse Mean Harmonic Curvature Flow, respectively.  
\end{example}

\begin{notation}
For now on we will denote by $S_{k,i}(\lambda)$ and $S_{k,i;j}(\lambda)$ the sum of all terms of $S_k(\lambda)$ which do not contain the factor $\lambda_i$ and the factors $\lambda_i$ and $\lambda_j$, respectively.
\end{notation}

\begin{lemma}\label{L 2.4}
For any $k\in\set{0,\ldots,n}$, $i\in\set{1,\dots,n}$ and $\lambda\in\rr^{n}$ we have
\begin{align}
\dfrac{\partial S_{k+1}}{\partial \lambda_i}&=S_{k, i},\label{P 2.4.1}
\\
S_{k+1}&=S_{k+1,i}+\lambda_iS_{k,i},\label{P 2.4.2}
\\
\sum_{i=1}^n S_{k,i}&=(n-k)S_k,\label{P 2.4.3}
\\
\sum_{i=1}^n \lambda_iS_{k,i}&=(k+1)S_{k+1},\label{P 2.4.4}
\\
\sum_{i=1}^{n}\lambda_i^2S_{k,i}&=S_1S_{k+1}-(k+2)S_{k+2},\label{P 2.4.5}
\\ \label{P 2.4.6}
\sum_{i=1}^nS_{k,i;j}&=(n-k-1)S_{k,j}.
\end{align}
\end{lemma}

\begin{proof}
For a proof of (\ref{P 2.4.1})-(\ref{P 2.4.5}) we refer \cite{huisken-sinestrari_1999}. Note that the last one follows by taking derivative with respect $\lambda_j$ on (\ref{P 2.4.3}).
\end{proof}

As an easy application of Lemma \ref{L 2.4} we obtain,

\begin{lemma}\label{L 2.5}
For any $k\in\set{0,\ldots,n}$, $i\in\set{1,\dots,n}$ and $\mu\in\rr^{n}$ we have
\begin{align}\label{2.5.1}
&\sum_{i=1}^n\dfrac{\partial Q_k}{\partial \lambda_i}=(n-k)-(n-k+1)\dfrac{Q_k}{Q_{k-1}},
\\\label{2.5.2}
&\sum_{i=1}^n\lambda_i\dfrac{\partial Q_k}{\partial \lambda_i}=Q_k,
\\\label{2.5.3}
&\sum_{i=1}^n\lambda_i^2\dfrac{\partial Q_k}{\partial \lambda_i}=(k+1)Q_k^2-(k+2)Q_{k+1}Q_k.
\end{align}
\end{lemma}

Another important result is the Newton inequality for the $S_k$ polynomials,

\begin{lemma}\label{L2.6}
For any $k\in\set{1,\ldots,n-1}$ and $\lambda\in\rr^n$ we have
\begin{align*}
    (k+1)\left(n-(k-1)\right)S_{k-1}S_{k+1}\leq k(n-k)S_{k}^2. 
\end{align*}
Equality holds if, and only if, all $\lambda_i$ are equal.
\end{lemma}

\begin{proof}
For a proof we refer to \cite{hardy_littlewood_polya_1952}. 
\end{proof}
Iterating Lemma \ref{L2.6} we also obtain,
\begin{corollary}\label{C 2.7}
For any $l,k\in\set{0,\ldots,n-1}$ such that $l\leq k$ and $\lambda\in\rr^n$ we have 
\begin{align}\label{lemma 2.6}
    Q_{k}\leq\dfrac{(l+1)(n-k)}{(k+1)(n-l)} Q_l.
\end{align}
\end{corollary}

\section{Gradient Estimates and applications}\label{sec3}

In this section we derive local gradient estimates for solutions to \eqref{I.2} and \eqref{I.3}. Note that locally, any solution of these equations we can be seen as a graph over neighborhood of a fixed point. 
\newline
Consequently, for a fixed point $F(x,t)=(x, u(x,t))$ and $F_0(x)=(x,u(x))$, the following equations holds
\begin{align}\label{4.1}
    u_t&=Q_k(Du,D^2u)\sqrt{1+|Du|^2},
    \\  \label{4.1.}
        Q_k(Du,D^2u)&=\dfrac{1}{\sqrt{1+|Du|^2}},
\end{align}
respectively. 
\newline
In addition, we consider the symmetrization of the Weingarten map, given by the matrix 
\begin{align}\label{4.0}
    A_{ij}=w^{-1}\left(u_{ij}-\dfrac{u_iu_ku_{kj}}{w(1+w)}-\dfrac{u_ju_ku_{ki}}{w(1+w)}+\dfrac{u_iu_ju_ku_lu_{kl}}{w^2(1+w)^2}\right),
\end{align}
where $w=\sqrt{1+|Du|^2}$ and the subindices denote derivatives with respect to the corresponding variable. It is well know fact that the eigenvalues of  the matrix $A_{ij}$ correspond to the principal curvatures of the graph of the function $u$. 
\newline
The method that we employ to obtain the estimate \eqref{T 34} was first used in \cite{caffarelli_nirenberg_spruck_1988} for deriving gradient estimates for Dirichlet Curvature PDEs. In \cite{holland_2014}, the author also used this method for deriving gradient estimate for the $(S_k)^{\frac{1}{k}}$-flow. 
\begin{remark}
Recall the function $S_k$ can be evaluated in a symmetric matrix $B$ by the relation	
	\begin{align*}
		\det(tI+B)=t^n+S_1(B)t^{n-1}+\ldots+S_{n}(B).
	\end{align*}
   Indeed, we choose $\alpha\subset \set{1,\ldots,n}$ and $|\alpha|$ denotes its cardinality. Then, we set $B[\alpha]$ to be the principal submatrix of $B$ in rows and columms index by $\alpha$, then it follows that  
   \begin{align*}
   	S_k(B)=\binom{n}{k}^{-1} \sum\limits_{\alpha\subset\nn, |\alpha|=k}\det(B[\alpha]).
   	\end{align*}
    It is relevant to remark that the properties stated in Section \ref{section:Properties} are valid for the eigenvalues of a given matrix (or equivalently, for diagonal matrices).
\end{remark}
The following lemma was inspired by \cite{mariel_2019}  and it is due to the shape of matrix $(A_{ij})$ after a change of coordinates for which $S_k(A)$ is easier to compute.  
\begin{lemma}\label{L 3.1}
If matrix $(A_{ij})$ has the form, 
\begin{align*}
A_{11}=\dfrac{u_{11}}{w^3}<0,\:A_{1j}=\dfrac{u_{1j}}{w^2}=A_{j1},\:A_{jj}=\dfrac{u_{jj}}{w}\mbox{ and }A_{ij}=A_{ji}=0\: (\mbox{for }2\leq i\neq j\leq n).
\end{align*}
 at some point $(x_0,t_0)$ or $x_0$ where its principal values lies in the cone $\Gamma_{k+1}$, then we have the following equations:
\begin{align}\label{4.1.10}
&\sum_{i=1}^n\dfrac{\partial Q_k}{\partial A_{ij}}A_{ij}=Q_k,
\\
&S_{l}(\tilde{\lambda})>0,\: \dfrac{\partial S_{l}}{\partial A_{jj}}(\tilde{\lambda})=S_{l-1}(\tilde{\lambda}|j)>0 \mbox{ for } j=2,\ldots n \mbox{ and }l=1,\ldots,k+1, 
\\ \label{4.1.17}
&S_l(\tilde{\lambda})=S_l(\tilde{\lambda}|j)+A_{jj}S_{l-1}(\tilde{\lambda}|j),
\\ \label{4.1.20}
&S_l(\tilde{\lambda})>S_l,
\\ \label{4.1.18}
&S_l(A)=S_l(\tilde{\lambda})+A_{11}S_{l-1}(\tilde{\lambda})-\sum_{j>1}A_{1j}^2S_{l-2}(\tilde{\lambda}|j),
\\ \label{4.1.19}
&S_l(A)=S_l(A|i)+A_{ii}S_{l-1}(A|i)-A_{1i}^2S_{l-2}(A|1i),
\\ \label{4.1.14}
&\dfrac{\partial Q_k}{\partial A_{11}}\geq \dfrac{n}{(k+1)(n-k)},\:\dfrac{\partial Q_k}{\partial A_{ii}}>0\mbox{ for }i>1,
\\ \label{4.1.16}
&\dfrac{\partial Q_k}{\partial A_{11}}\geq \dfrac{n}{(n-k)^2(k+1)}\sum_{i\geq 1}\dfrac{\partial Q_k}{\partial A_{ii}}.
\end{align}
\end{lemma}
\begin{remark}
	Here we are using the following notation:  $S_{l,ij}=\dfrac{\partial S_l}{\partial u_{ij}}$, $\tilde{\lambda}=\mbox{diag}(A_{22},\ldots,A_{nn})$,  $S_k(B|i)$ or $S_k(B|ij)$ means that the $i$-th row and $i$-th column resp. $i,j$-th row and $i,j$-th column are omitted from a matrix $B$.
\end{remark}

\begin{proof}
Under this setting we have  
\begin{align*}
Q_{k,ij}=\left(\dfrac{S_{k+1}}{S_{k}}\right)_{ij}=\dfrac{S_{k+1,ij}S_k-S_{k+1}S_{k,ij}}{S_k^2}.
\end{align*}
Then it follows,
\begin{align*}
&Q_{k,11}=\dfrac{\partial Q_k}{\partial A_{11}}\dfrac{\partial A_{11}}{\partial u_{11}}=\dfrac{\partial Q_k}{\partial A_{11}}\dfrac{1}{w^{3}},
\\
&Q_{k,1i}=\dfrac{\partial Q_k}{\partial A_{1i}}\dfrac{\partial A_{1i}}{\partial u_{1i}}=\dfrac{\partial Q_k}{\partial A_{1i}}\dfrac{1}{w^{2}}=Q_{k,i1},\:i\neq 1
\\
&Q_{k,ii}=\dfrac{\partial Q_k}{\partial A_{ii}}\dfrac{\partial A_{ii}}{\partial u_{ii}}=\dfrac{\partial Q_k}{\partial A_{ii}}\dfrac{1}{w}
\\
&Q_{k,ij}=\dfrac{\partial Q_k}{\partial A_{ij}}\dfrac{\partial A_{ij}}{\partial u_{ij}}=\dfrac{\partial Q_k}{\partial A_{ij}}\dfrac{1}{w}=0=Q_{k,ji}\: (\mbox{for }2\leq i\neq j\leq n).
\end{align*}
We note that Equation (\ref{4.1.10}) follows by Equation (\ref{2.5.2}) together with the formula 
$\sum\limits_{i,j}\dfrac{\partial S_k}{\partial A_{ij}}A_{ij}=(k+1)S_{k+1}$. 
\newline
Furthermore, since $A_{11}<0$, it follows that $S_{l}(\tilde{\lambda})>0$ for $l=1$. Note that this fact implies that $\dfrac{\partial S_{l}}{\partial A_{jj}}(\tilde{\lambda})=S_{l-1}(\tilde{\lambda}|j)>0$.  Then, by iterating this process, we can get the same result for $j=2,\ldots n$ and $l=1,\ldots,k+1$. Consequently, by the shape of the matrix $A_{ij}$ and the properties given in Lemma \ref{L 2.4}, equations \eqref{4.1.17}-\eqref{4.1.19} hold.  
\\
Now we analyze the diagonal terms $\dfrac{\partial Q_k}{\partial A_{ii}}$. For $i=1$, it follows
\begin{align*}
\dfrac{\partial Q_k}{\partial A_{11}}&=\dfrac{1}{S_k^2}\left(\dfrac{\partial S_{k+1}}{\partial A_{11}}S_k-S_{k+1}\dfrac{\partial S_k}{\partial A_{11}}\right)=\dfrac{1}{S_k^2}\left(S_k(\tilde{\lambda})S_k-S_{k+1}S_{k-1}(\tilde{\lambda})\right).
\end{align*}
Then, by combining the above Equation with Equation (\ref{4.1.17}), we obtain 
\begin{align*}
\dfrac{\partial Q_k}{\partial A_{11}}=\dfrac{1}{S_k^2}\left(S_k(\tilde{\lambda})^2-S_{k-1}(\tilde{\lambda})S_{k+1}(\tilde{\lambda})+\sum_{j>1}A_{1j}^2(S_{k-1}(\tilde{\lambda}|j)^2-S_k(\tilde{\lambda}|j)S_{k-2}(\tilde{\lambda}|j)  )  \right).
\end{align*}
Note that by Lemma \ref{L2.6}, the term in the sum is non-negative, and therefore we drop it form the inequality. Then, by  Equation \eqref{4.1.20}, it follows that
\begin{align*}
\dfrac{\partial Q_k}{\partial A_{11}}\geq \dfrac{n}{(k+1)(n-k)}\left(\dfrac{S_{k}(\tilde{\lambda})}{S_k}\right)^2\geq \dfrac{n}{(k+1)(n-k)}.
\end{align*}

Now we show that $\dfrac{\partial Q_k}{\partial A_{ii}}\geq 0$ for $i>1$. Indeed, by Equation \eqref{4.1.19} we have,
\begin{align*}
\dfrac{\partial Q_k}{\partial a_{ii}}&=\dfrac{1}{S_k^2}\left(S_k(A|i)^2-S_{k+1}(A|i)S_{k-1}(A|i)\right)
\\
&\hspace{4.3mm}+\dfrac{A_{1i}^2}{S_k^2}\left[S_{k-1}(A|i)S_{k-1}(A|1i)-S_k(A|i)S_{k-2}(A|1i)\right].
\end{align*}
Now we use Equation \eqref{4.1.18} on each $S_l(a|i)$ to obtain,
\begin{align*}
\dfrac{\partial Q_k}{\partial A_{ii}}&=\dfrac{1}{S_k^2}\left(S_k(A|i)^2-S_{k+1}(A|i)S_{k-1}(A|i)\right)+\dfrac{A_{1i}^2}{S_k^2}\left(S_{k-1}(\tilde{\lambda}|i)^2-S_k(\tilde{\lambda}|i)S_{k-2}(\tilde{\lambda}|i)\right)
\\
&\hspace{.5cm}+\dfrac{A_{1i}^2}{S_k^2}\sum_{j>1,j\neq i}A_{1j}^2\left[S_{k-2}(\tilde{\lambda}|ij)S_{k-2}(\tilde{\lambda}|ij)-S_{k-3}(\tilde{\lambda}|ij)S_{k-1}(\tilde{\lambda}|ij)\right].
\end{align*}
We note that these three terms are non-negative by Lemma \ref{L2.6}. Finally, for the sum of the diagonal terms we have,
\begin{align*}
\sum_{i\geq 1}\dfrac{\partial Q_k}{\partial A_{ii}}&=
\dfrac{1}{S_k}\left(S_k(\tilde{\lambda})+\sum_{i>1}\left[S_k(\tilde{\lambda}|i)+a_{11}S_{k-1}(\tilde{\lambda}|i)-\sum_{j>1,j\neq i}a_{1j}^2S_{k-2}(\tilde{\lambda}|ij)\right]\right)
\\ 
&\hspace{-7mm}-\dfrac{Q_k}{S_k}\left(S_{k-1}(\tilde{\lambda})+\sum_{i>1}\left[S_{k-1}(\tilde{\lambda|i})+a_{11}S_{k-2}(\tilde{\lambda}|i)-\sum_{j>1,j\neq i}a_{1j}^2S_{k-3}(\tilde{\lambda}|ij)\right]\right).
\end{align*}
By applying Equation \eqref{P 2.4.3} on each term of the form,
\begin{align*}
\sum_{i>1}S_l(\tilde{\lambda}|i)=(n-1-l)S_l(\tilde{\lambda}),\:l=k-2,k-1,k
\end{align*}
together with Equation \eqref{P 2.4.6} on the terms,
\begin{align*}
\sum_{i>1}S_l(\tilde{\lambda}|ij)=(n-1-l-1)S_l(\tilde{\lambda}|j),\:l=k-3,k-2,
\end{align*}
it will follow that,
\begin{align*}
\sum_{i\geq 1}\dfrac{\partial Q_k}{\partial a_{ii}}&=\dfrac{(n-k)}{S_k}\left(S_k(\tilde{\lambda})+a_{11}S_{k-1}(\tilde{\lambda})-\sum_{j>1}a_{1j}^2S_{k-2}(\tilde{\lambda}|j)\right)
\\ 
&\hspace{.1cm}-(n-k+1)\dfrac{Q_k}{S_k}\left(S_{k-1}(\tilde{\lambda})+a_{11}S_{k-2}(\tilde{\lambda})-\sum_{j>1}a_{1j}^2S_{k-3}(\tilde{\lambda}|j)\right).
\end{align*}
Finally, by Equation \eqref{4.1.18}, we have 
\begin{align*}
\sum_{i\geq 1}\dfrac{\partial Q_k}{\partial a_{ii}}&=(n-k)-(n-k+1)\dfrac{Q_k}{Q_{k-1}}\leq (n-k),
\end{align*}
and after combining this with Equation (\ref{4.1.14}), Equation \eqref{4.1.16} will hold. 
\end{proof}

\begin{proposition}\label{P 3.2}
Let $u:\Omega\to\rr$ be a solution to Equation (\ref{4.1}), where $\Omega=B_r(0)\times(0,T]$ with $r>0$. Assume that the principal curvatures of the graph of $u$ lies in the cone $\Gamma_{k+1}$ for each $t$. Then, for $t>0$, it holds 
\begin{align}\label{4.1.0}
    |Du(0,t)|\leq \exp\left(K+\dfrac{KMT}{rt}\dfrac{MrT}{t}+\dfrac{M^2}{t}+\dfrac{(T+1)r^2}{t}+\dfrac{TM^2}{tr^2}+\dfrac{TM}{tr}\right),
\end{align}
where $M=\sup\limits_{\Omega}u$ and $K=K(k,n)$. 
\end{proposition}

\begin{proof}
The proof is very similar to the given for Theorem 1.1  in \cite{holland_2014}. For this reason we will use the same notation. 
\newline
Note that by a translation on the $x_{n+1}$-axis, we may assume that $u>0$.  In addition, we take the matrix $A$ given in \eqref{4.0}, and by a slight abuse of notation we set, 
\begin{align*}
    Q_k(A)=\dfrac{S_{k+1}(A)}{S_k(A)},
\end{align*}
where each  function is evaluated in the eigenvalues of $A$. Consequently, Equation \eqref{4.1} becomes 
\begin{align}{\label{4.1.1}}
    Q_k(A)=\dfrac{u_t}{w},\: w=\sqrt{1+|Du|^2}.
\end{align}

Now we define the test function $G:\Omega\times\Sp^{n-1}\to\rr$ given by 
\begin{align*}
    G(x,t,\xi)=t\rho(x)\varphi(u)\ln(u_{\xi}),
\end{align*}
here $\xi$ denote the direction derivative vector,
\begin{align} \label{test 1}
    \rho(x)=1-\dfrac{|x|^2}{r^2}\mbox{ and }\varphi(u)=1+\dfrac{u}{M},
\end{align}
where $M=\sup\limits_{\Omega}u$. 
\newline
Since $\rho$ vanishes at $\partial B_r(0)$, we may suppose that $G$ attains its maximum at some point $(x_0,t_0)$ with $t_0>0$ and $|x_0|<r$. Moreover, after a rotation,  we may choose $\xi=\eps_1$ where $\eps_i$ denotes the canonical euclidean base of $\rr^{n+1}$. 
\newline
Then at $(x_0,t_0)$ the following equations holds,
\begin{align}\label{4.1.3}
    0&=(\ln G)_i=\dfrac{\rho_i}{\rho}+\dfrac{\varphi'}{\varphi}u_i+\dfrac{u_{1i}}{u_1\ln u_1},
\end{align}
 \begin{align} \label{4.1.2}
     0&\geq(\ln G)_{ij}
     \\ \notag
     &=\dfrac{\rho_{ij}}{\rho}+\dfrac{\varphi'}{\varphi}u_{ij}+\dfrac{\varphi'}{\varphi\rho}(\rho_iu_j+\rho_ju_i)+\dfrac{u_{ij1}}{u_1\ln u_1}-\left(1+\dfrac{2}{\ln u_1}, \right)\dfrac{u_{1i}u_{1j}}{u_1^2\ln u_1},
\end{align}
    \begin{align}  \label{4.1.7}
    0&\leq (\ln G)_t=\dfrac{\varphi'}{\varphi}u_t+\dfrac{u_{1t}}{u_1\ln u_1}+\dfrac{1}{t}.
\end{align}
Finally, we change the coordinates on $\Omega$ in such a way that the following equations hold,
\begin{align}\label{4.1.21}
&u_i(x_0,t_0)=0,\mbox{ for }i\neq 1;\:u_{ij}(x_0,t_0)=0,\mbox{ for }i\neq j \mbox{ and }i,j\geq 2;
\\ \notag
&u_{22}(x_0,t_0)\geq\ldots\geq u_{nn}(x_0,t_0).
\end{align}
We claim that  $u_{11}<0$ at $(x_0,t_0)$. Indeed, at this point we have 
\begin{align*}
1-\dfrac{u_1^2}{w(w+1)}=\dfrac{1}{w}.
\end{align*}
Then, it follows that
\begin{align*}
A_{11}=\dfrac{u_{11}}{w^3},\:A_{1j}=\dfrac{u_{1j}}{w^2}=A_{j1},\:A_{jj}=\dfrac{u_{jj}}{w}\mbox{ and }A_{ij}=A_{ji}=0\: (\mbox{for }2\leq i\neq j\leq n).
\end{align*}
Furthermore, by Equation  \eqref{4.1.3}, it follows that
\begin{align}\label{4.1.8}
&\dfrac{u_{11}}{u_1\ln u_1}=-\dfrac{\rho_1}{\rho}-\dfrac{\varphi'}{\varphi}u_1,
\\ \label{4.1.9}
&\dfrac{u_{1i}}{u_1\ln u_1}=-\dfrac{\rho_i}{\rho} \:\:(\mbox{for }i\geq 2). 
\end{align}
In what follows all quantities are evaluated at $(x_0,t_0)$ and also we assume that $G(x_0,t_0)$ is big enough such that 
\vskip -5mm
\begin{align*}
    G=t\rho\varphi\ln u_1\geq 16\dfrac{MT}{r}.
\end{align*}
Therefore, we obtain
\begin{align}\label{4.1.12}
u_1\geq \dfrac{8M}{r\rho}\mbox{ and }\:16\dfrac{MT}{r}\geq\dfrac{8\varphi t}{\varphi'r},
\end{align}
which implies
\begin{align}\label{4.1.4}
    \abs{\dfrac{\rho_j}{\rho}}<\dfrac{1}{4}\dfrac{\varphi'}{\varphi}u_1,  
\end{align}
for all $j$. Finally, by combining \eqref{4.1.8} and \eqref{4.1.4}, it holds
\begin{align}\label{4.1.11}
    u_{11}=u_1\ln u_1\left(-\dfrac{\rho_1}{\rho}-\dfrac{\varphi'}{\varphi}u_1 \right)\leq- u_1^2\dfrac{\varphi'}{2\varphi}\ln u_1<0.
\end{align}
Hence, $A_{11}<0$ and we can use the equations from Lemma \ref{L 3.1}.

Then, by Equation \eqref{4.1.2}, it follows
\begin{align} \label{2 terminos}
    0\geq Q_{k,ij}(\ln G)_{ij}&=\underbrace{Q_{k,ij}\left(\dfrac{\rho_{ij}}{\rho}+\dfrac{\varphi'}{\varphi}u_{ij}+\dfrac{\varphi'}{\varphi\rho}(\rho_iu_j+\rho_ju_i) \right)}_{=B}
    \\ \notag
    &\hspace{4.5mm}+\underbrace{Q_{k,ij}\left(\dfrac{u_{ij1}}{u_1\ln u_1}-\left(1+\dfrac{2}{\ln u_1} \right)\dfrac{u_{1i}u_{1j}}{u_1^2\ln u_1}\right)}_{=C}.
   \end{align}
We start with the term $B$ in \eqref{2 terminos}. By  Equation \eqref{4.1.21}, we can split $B$ in the following way 
\begin{align}\label{B flow}
    B&=Q_{k,11}\left(\dfrac{\rho_{11}}{\rho}+\dfrac{\varphi'}{\varphi}u_{11}+\dfrac{2\varphi'}{\varphi\rho}\rho_1u_1 \right)+\sum_{i>1}Q_{k,ii}\left(\dfrac{\rho_{ii}}{\rho}+\dfrac{\varphi'}{\varphi}u_{ii}\right)
    \\ \notag
    &\hspace{4.5mm}+ 2\sum_{j>1}Q_{k,1j}\left(\dfrac{\rho_{1j}}{\rho}+\dfrac{\varphi'}{\varphi}u_{1j}+\dfrac{\varphi'}{\varphi\rho}\rho_ju_1 \right).
\end{align}
Then, by equations \eqref{4.1.10} and \eqref{4.1.1}, the term $\sum\limits_{i,j}Q_{k,ij}\dfrac{\varphi'}{\varphi}u_{ij}$ in \eqref{B flow} satisfies,
\begin{align}\label{3.1.}
\dfrac{\varphi'}{\varphi}\sum_{i,j}Q_{k,ij}u_{ij}=\dfrac{\varphi'}{\varphi}\left(Q_{k,11}u_{11}+2\sum_{i>1}Q_{k,1i}u_{1i}+Q_{k,ii}u_{ii}\right)=\dfrac{\varphi'}{\varphi}\dfrac{u_t}{w}.
\end{align} 

Moreover, by equations \eqref{4.1.9} and \eqref{4.1.18},  the last term in \eqref{B flow} can be seen as
\begin{align*}
&2\sum_{j> 1}Q_{k,1j}\dfrac{\varphi'\rho_j}{\varphi\rho}u_1=-2\dfrac{u_1^2\ln u_1}{\rho}\dfrac{\varphi'}{\varphi}\sum_{j>1}\dfrac{\partial Q_k}{\partial A_{1j}}A_{1j}
\\
&=2\dfrac{u_1^2\ln u_1}{\rho}\dfrac{\varphi'}{\varphi}\dfrac{S_k(S_{k+1}(A|1)+A_{11}S_{k}(A|1)-S_{k+1})-S_{k+1}(S_k(A|1)+A_{11}S_{k-1}(A|1)-S_k)}{S_k^2}
\\
&=2\dfrac{u_1^2\ln u_1}{\rho}\dfrac{\varphi'}{\varphi}\dfrac{S_kS_{k+1}(\lambda)-S_{k+1}S_k(\lambda)}{S_k^2},
\end{align*}
in the last line we use Equation \eqref{4.1.20}. Recall that $S_l(\lambda)$ denote the $l$-elemental symmetric polynomial evaluated in the diagonal matrix $\lambda=(A_{ii})_{i\geq 1}$. 
\newline
Then, it follows that  
\begin{align*}
2\sum_{j> 1}Q_{k,1j}\dfrac{\varphi'\rho_j}{\varphi\rho}u_1&=2\dfrac{u_1^2\ln u_1}{\rho}\dfrac{\varphi'}{\varphi}\dfrac{S_k(\lambda)}{S_k}(Q_k(\lambda)-Q_k).
\end{align*}
Following an idea of \cite{mariel_2019}, we may use a first order Taylor expansion on $Q_k(\lambda)-Q_k$ around $\lambda$ to see that,
\begin{align*}
Q_k=Q_k(\lambda)+\sum_{j>1}\dfrac{\partial Q_k}{\partial A_{1j}}(\lambda)A_{1j}+\sum_{j>1}\dfrac{\partial^2Q_k}{\partial A_{1j}\partial a_{1j}}(\eta)A_{1j}^2\leq Q_k(\lambda), 
\end{align*} 
where the inequality follows from $\dfrac{\partial Q_k}{\partial A_{1j}}(\lambda)=0$ and the error term is non positive by concavity of $Q_k$. Hence, the whole term is non-negative and we may drop it from \eqref{B flow}.
\newline
For the term $2Q_{k,11}\dfrac{\varphi'\rho_1}{\varphi\rho}u_1$, by equations  \eqref{4.1.14} and \eqref{4.1.8}, we have
\begin{align*}
2Q_{k,11}\dfrac{\varphi'\rho_1}{\varphi\rho}u_1&=-2Q_{k,11}\dfrac{\varphi'}{\varphi}\left(\dfrac{u_{11}}{u_1\ln u_1}+\dfrac{\varphi'}{\varphi}u_1 \right)u_1\geq-\dfrac{\varphi'^2u_1^2}{2\varphi^2w^3}\dfrac{\partial Q_k}{\partial A_{11}}
\\
&\geq -\dfrac{u_1^2}{2M^2w^3}\sum_{i\geq 1}\dfrac{\partial Q_k}{\partial A_{ii}}.
\end{align*}
Finally, for the term $\sum\limits_{i,j}Q_{k,ij}\dfrac{\rho_{ij}}{\rho}$ it follows that 
 \begin{align}\label{3.2.}
   \sum_{i,j}Q_{k,ij}\dfrac{\rho_{ij}}{\rho}&=-\dfrac{2}{r^2\rho}\sum_{i\geq 1}Q_{k,ii}
   \\ \notag
   &=-\dfrac{2}{r^2\rho}\left(\dfrac{1}{w^3}\dfrac{\partial Q_k}{\partial a_{11}}+\dfrac{1}{w}\sum_{i\geq 1}\dfrac{\partial Q_k}{\partial a_{ii}}\right)\geq -\dfrac{2}{wr^2\rho}\sum_{i\geq 1}\dfrac{\partial Q_k}{\partial a_{ii}}.
\end{align}
Consequently we obtain,
\begin{align}\label{3.2 B}
B\geq \dfrac{\varphi'}{\varphi}\dfrac{u_t}{w}-\dfrac{2}{wr^2\rho}\sum_i\dfrac{\partial Q_k}{\partial a_{ii}}-\dfrac{u_1^2}{2M^2w^3}\sum_{i}\dfrac{\partial Q_k}{\partial a_{ii}}.
\end{align}
Now we estimate the term $C$ in Equation \eqref{2 terminos}. By differentiating Equation \eqref{4.1.1} in the $\eps_1$-direction, we obtain 
\begin{align*}
\dfrac{u_{t1}}{w}-\dfrac{u_t}{w^3}u_1u_{11}&=\dfrac{\partial Q_k}{\partial A_{ij}}A_{ij,1}   
\\
&=\dfrac{\partial Q_k}{\partial A_{11}}A_{11,1}+2\dfrac{\partial Q_k}{\partial A_{i1}}A_{i1,1}+\dfrac{\partial Q_k}{\partial A_{ii}}a_{ii,1},
\end{align*}
where $A_{ij, 1}$ denotes $\dfrac{\partial A_{ij}}{\partial x_1}$ and
\begin{align*}
A_{11,1}&=\dfrac{u_{111}}{w^3}-\dfrac{3u_1}{w^5}u_{11}^2-\dfrac{2u_1}{w^3(1+w)}\sum_{j>1}u_{1j}^2, 
\\
A_{1i,1}&=\dfrac{u_{1i1}}{w^2}-\dfrac{2u_1}{w^4}u_{11}u_{1i}-\dfrac{u_1}{w^2(1+w)}u_{1i}u_{ii}-\dfrac{u_1}{w^3(w+1)}u_{11}u_{i1},
\\ 
A_{ii,1}&=\dfrac{u_{ii1}}{w}-\dfrac{u_1}{w^3}u_{11}u_{ii}-\dfrac{2u_1}{w^2(1+w)}u_{1i}^2.
\end{align*}
Then, we have the following equation
\begin{align*}
&Q_{k,ij}u_{ij1}=\dfrac{\partial Q_k}{\partial A_{11}}\dfrac{u_{111}}{w^3} + 2\sum_i\dfrac{\partial Q_k}{\partial A_{1i}}\dfrac{u_{1i1}}{w^2} + \sum_{i}\dfrac{\partial Q_k}{\partial A_{11}}\dfrac{u_{ii1}}{w}
\\
&=\dfrac{u_{t1}}{w}+\dfrac{\partial Q_k}{\partial A_{11}}\left(\dfrac{2u_{1}}{w^5}u_{11}^2+\dfrac{2u_1}{w^3(1+w)}\sum_{j}u_{1j}^2\right)+\sum_{i}\dfrac{\partial Q_k}{\partial A_{ii}}\dfrac{2u_1}{w^2(1+w)}u_{1i}^2
\\
&+2\sum_{i}\dfrac{\partial Q_k}{\partial A_{1i}}\left(\dfrac{u_1}{w^4}u_{11}u_{1i}+\dfrac{u_1}{w^2(1+w)}u_{1i}u_{ii}+\dfrac{u_1}{w^3(w+1)}u_{11}u_{i1}\right), 
\end{align*}
where in the last equality we use Equation \eqref{4.1.10}. After replacing and combining this last Equation in $C$ with Equation \eqref{4.1.8}  we obtain,
\begin{align} \label{C flow}
C&=Q_{k,ij}\left(\dfrac{u_{ij1}}{u_1\ln u_1}-\left(1+\dfrac{2}{\ln u_1} \right)\dfrac{u_{1i}u_{1j}}{u_1^2\ln u_1}\right)
\\ \notag
&\geq\dfrac{u_{t1}}{wu_1\ln u_1}+\dfrac{\partial Q_k}{\partial A_{11}}\left(\dfrac{2u_{1}^2}{w^2}-1-\dfrac{2}{\ln u_1}\right)\dfrac{u_{11}^2}{w^3u_1^2\ln u_1}
\\ \notag
&\hspace{4.5mm}+\dfrac{2}{w^2(1+w)\ln u_1}\sum_{i>1}\dfrac{\partial Q_k}{\partial A_{1i}}u_{1i}u_{ii}.
\end{align}
For the term third term in Equation \eqref{C flow} we have,
\begin{align*}
\dfrac{\partial Q_k}{\partial A_{1i}}u_{1i}u_{ii}&=-\dfrac{S_{k-1}(A|1i)S_k-S_{k-2}(A|1i)S_{k+1}}{S_k^2}A_{1i}u_{1i}u_{ii}
\\
&=\dfrac{u_{1i}^2}{wS_k^2}\left(-S_{k}^2\dfrac{\partial Q_k}{\partial A_{11}}+S_kS_k(A|1i)-S_{k+1}S_{k-1}(A|1i)\right).
\end{align*}
We note that the difference of the last two terms in the last line is non-negative. Indeed, if we consider
\begin{align*}
\dfrac{\partial Q_k}{\partial a_{ii}}(A|1)=\dfrac{S_k(A|1i)S_k(A|1)-S_{k+1}(A|1)S_{k-1}(A|1i)}{S_k(A|1)^2},
\end{align*}
it follows that 
\begin{align*}
&S_kS_k(A|1i)-S_{k+1}S_{k-1}(A|1i)
\\
&=\dfrac{\partial Q_k}{\partial A_{jj}}(A|1)S_{k}(A|1)S_k+\dfrac{S_{k-1}(A|1j)S_{k+1}(A|1)S_k}{S_k(A|1)}-S_{k+1}S_{k-1}(A|1j)
\\
&\geq S_{k-1}(A|1j)S_k(Q_k(A|1)-Q_k),
\end{align*}
which is non-negative from $A_{11}<0$ and the concavity of $Q_k$. 
\newline
Then, for the whole term we have
\begin{align}\label{3.5.}
\dfrac{2}{w^2(1+w)\ln u_1}\sum_{i}\dfrac{\partial Q_k}{\partial A_{1i}}u_{1i}u_{ii}&\geq-\dfrac{2}{w^3(1+w)}\sum_{i>1}\dfrac{u_{1i}^2}{\ln u_1}\dfrac{\partial Q_k}{\partial A_{11}}
\\ \notag
&\geq-2\dfrac{u_1^2\ln u_1}{w^3(1+w)}\sum_{i>1}\dfrac{\rho_i^2}{\rho^2}\sum_{j\geq 1}\dfrac{\partial Q_k}{\partial a_{jj}}
\\ \notag
&\geq -8\dfrac{u_1^2\ln u_1}{w^3(1+w)\rho^2}\sum_{j\geq 1}\dfrac{\partial Q_k}{\partial a_{jj}}.
\end{align}
Finally, for the second term in \eqref{C flow} we use Equation \eqref{4.1.11} to obtain,
\begin{align*}
\dfrac{\partial Q_k}{\partial a_{11}}\left(\dfrac{2u_{1}^2}{w^2}-\left(1+\dfrac{2}{\ln u_1}\right)\right)\dfrac{u_{11}^2}{w^3u_1^2\ln u_1}&\geq \dfrac{\partial Q_k}{\partial a_{11}}\left(\dfrac{2u_{1}^2}{w^2}-\left(1+\dfrac{2}{\ln u_1}\right)\right)\dfrac{u_1^2\ln u_1}{w^3}\dfrac{\varphi'^2}{4\varphi^2}
\\
&\geq\dfrac{\partial Q_k}{\partial a_{11}}\left(\dfrac{2u_{1}^2}{w^2}-\left(1+\dfrac{2}{\ln u_1}\right)\right)\dfrac{u_1^2\ln u_1}{w^3}\dfrac{1}{16M^2}
\\
&\geq\dfrac{C_2(k,n)}{64M^2}\dfrac{u_1^2\ln u_1}{w^3}\sum_{i\geq 1}\dfrac{\partial Q_k}{\partial a_{ii}},
\end{align*}
where $C_2(k,n)$ is the constant given in \eqref{4.1.16}.
\newline
Note that we have chosen $c_0$ big enough such that $u_1>c_0$ and $\ln u_1>0$. Hence, it follows that
\begin{align}\label{3.2.A}
C\geq \sum_{i\geq 1}\dfrac{\partial Q_k}{\partial a_{ii}}\left(\dfrac{C_2(k,n)}{64M^2}\dfrac{u_1^2\ln u_1}{w^3}-C\dfrac{u_1^2\ln u_1}{w^3(w+1)}\dfrac{4}{\rho^2}\right)+\dfrac{u_{t1}}{wu_1\ln u_1}.
\end{align}
Combining  the lower bounds from \eqref{3.2 B} and \eqref{3.2.A} we get,
\begin{align*}
0&\geq\dfrac{u_{t1}}{wu_1\ln u_1}+u_t\dfrac{\varphi'}{w\varphi}
\\
&+\sum_{i}\dfrac{\partial Q_k}{\partial a_{ii}}\left(-\dfrac{4r}{\rho w^3}\dfrac{u_1^2}{M}-\dfrac{2}{wr^2\rho}+\dfrac{C_2(k,n)}{64M^2}\dfrac{u_1^2\ln u_1}{w^3}-C\dfrac{u_1^2\ln u_1}{w^3(w+1)}\dfrac{4}{\rho^2}\right).
\end{align*}
Moreover, by Equation \eqref{4.1.7}, we have 
\begin{align}\label{4.1.13}
\dfrac{\varphi'}{\varphi}u_t+\dfrac{u_{t1}}{u_1\ln u_1}\geq-\dfrac{1}{t}.
\end{align}
Then, after using Equation \eqref{4.1.13} and multiplying by $w^2\left(\sum\limits_{i\geq 1}\dfrac{\partial Q_k}{\partial a_{ii}}\right)^{-1}$, we obtain 
\begin{align*}
M^2\left(C\dfrac{\ln u_1}{(w+1)}\dfrac{4}{\rho}+\dfrac{k+1}{n-k}\dfrac{w^2\rho}{u_1^2t}+\dfrac{4r}{M}+\dfrac{2w^2}{u_1^2r^2}\right)\geq \dfrac{C_2}{64}\rho\ln u_1.
\end{align*}
Consequently, we obtain  
\begin{align*}
M^2\left(C\dfrac{4r}{M}+\left(1+\dfrac{r^2}{M^2} \right)\dfrac{1}{c_3t}+\dfrac{4r^2}{M^2}+\dfrac{1}{r^2}\left(1+\dfrac{r}{M}\right)\right)\geq \dfrac{c_2}{64}\rho\ln u_1,
\end{align*}
here we use the assumption \eqref{4.1.12}. Note that this implies,
\begin{align*}
\ln u_1\rho\leq K\left(Mr+\dfrac{1}{t}(M^2+r^2)+r^2+\dfrac{M^2}{r^2}+\dfrac{M}{r}\right),  
\end{align*}
for a universal constant $K=K(k,n)$. Then, we have
\begin{align*}
\ln| Du(0,t)|&\leq \dfrac{t_0}{t} \dfrac{\varphi(u(x_0,t_0))}{\varphi(u(0,t))}\dfrac{\rho(x_0)\ln u_1(x_0,t_0)}{\rho(0)}
\\
&\leq K\left(\dfrac{MrT}{t}+\dfrac{M^2}{t}+\dfrac{r^2}{t}+\dfrac{Tr^2}{t}+\dfrac{TM^2}{tr^2}+\dfrac{TM}{tr}\right).
\end{align*} 
Since we  assumed that $u_1\geq c_0$ and $G(x_0,t_0)\geq 16\dfrac{MT}{r}$, we finally obtain 
\begin{align*}
    |Du(0,t)|\leq \mbox{exp}\left(K+\dfrac{KMT}{rt}\dfrac{MrT}{t}+\dfrac{M^2}{t}+\dfrac{(T+1)r^2}{t}+\dfrac{TM^2}{tr^2}+\dfrac{TM}{tr}\right).
\end{align*}
\end{proof}

Recall that a $Q_k$-translator is a surface that evolves by translations with unit speed, hence we may use the same method to obtain a local gradient estimate for graphical solutions to Equation \eqref{4.1.}.

\begin{theorem}\label{T 34}
Let $r>0$ and $u\in\mathcal{C}^3\left(B(0,r)\right)$ be a solution \eqref{4.1} such that the principal curvatures of $\mbox{graph}(u)$ lies in $\Gamma_{k+1}$. Then it holds 
\begin{align*}
    |Du(0)|\leq\exp\left(\dfrac{CM}{r}+\dfrac{CM^2}{r^2}\right),
\end{align*}
where $M=\sup\limits_{B(0,r)}u$ and $C=C(k,n)$. 
\end{theorem}
\begin{proof}
This proof is very similar from the given in Proposition \ref{P 3.2}, for this reason we only point out the main differences from it. 
\newline
First, we note that equation \eqref{4.1.} can be written as
\begin{align}\label{4.2}
Q_k(A)=\dfrac{1}{w},
\end{align}
where $A$ is the matrix given in \eqref{4.0}.  
\newline
Secondly, we use the same test function $G(x,\xi)$ given in \eqref{test 1} without the time factor.  We also change the cut off function $\rho$ by
\begin{align*}
   \rho(x)=r^2-|x|^2.
\end{align*}
As before we may assume that the maximum of $G$ is reached at some point $x_0\in B_r(0)$. We also apply the same change of coordinates as we did before. Now, if we want to use equations from Lemma \ref{L 3.1}, we need to ensure that $u_{11}<0$ at $x_0$.  For this purpose, we assume  
\begin{align*}
G=\rho\varphi\ln u_1\geq 16rM.
\end{align*}
Then, it follows that $u_1\geq \dfrac{8rM}{\rho}$ and $\dfrac{\varphi'}{\varphi}\geq\dfrac{1}{2M}$, which also implies
\begin{align*}
\abs{\dfrac{\rho_j}{\rho}}<\dfrac{2r}{\rho}\leq \dfrac{\varphi'}{2\varphi}u_1.
\end{align*}
Finally we get,
\begin{align*}
u_{11}=u_1\ln u_1\left(-\dfrac{\rho_1}{\rho}-\dfrac{\varphi'}{\varphi}u_1\right)\leq -u_1^2\dfrac{\varphi'}{2\varphi}\ln u_1<0.
\end{align*} 
Note that we again will get the terms $B$ and $C$ from Equation \eqref{2 terminos}, which we now analyze in this configuration. 
\newline
We start with $B$ and note that the only terms that change are equations \eqref{3.1.} and \eqref{3.2.}. In this case, we have
\begin{align*}
Q_{k,ij}\dfrac{\varphi'}{\varphi}u_{ij}=\dfrac{\varphi'}{\varphi}Q_k=\dfrac{\varphi'}{w\varphi}.
\end{align*}
and 
\begin{align*}
   \sum_{i,j}Q_{k,ij}\dfrac{\rho_{ij}}{\rho}=-\dfrac{2}{\rho}\sum_{i\geq 1}Q_{k,ii}=-\dfrac{2}{\rho}\left(\dfrac{1}{w^3}\dfrac{\partial Q_k}{\partial A_{11}}+\dfrac{1}{w}\sum_{i\geq 1}\dfrac{\partial Q_k}{\partial A_{ii}}\right)\geq -\dfrac{2}{w\rho}\sum_{i\geq 1}\dfrac{\partial Q_k}{\partial A_{ii}}.
\end{align*}
Therefore, 
\begin{align*}
    B\geq \dfrac{\varphi'}{\varphi}\dfrac{1}{w}-\dfrac{2}{w\rho}\sum_i\dfrac{\partial Q_k}{\partial A_{ii}}-\dfrac{u_1^2}{2M^2w^3}\sum_{i}\dfrac{\partial Q_k}{\partial A_{ii}}. 
\end{align*}
For the term $C$, we only need to estimate the therm $Q_{k,ij}u_{ij1}$.  We observe that
\begin{align*}
Q_{k,ij}u_{ij1}&=\dfrac{\partial Q_k}{\partial A_{11}}\dfrac{u_{111}}{w^3} + 2\sum_i\dfrac{\partial Q_k}{\partial A_{1i}}\dfrac{u_{1i1}}{w^2} \sum_{i}\dfrac{\partial Q_k}{\partial A_{11}}\dfrac{u_{ii1}}{w}
\\
&=\underbrace{-\dfrac{u_1u_{11}}{w^3}+\dfrac{u_1u_{11}}{w^2}Q_k}_{=0}+\dfrac{\partial Q_k}{\partial A_{11}}\left(\dfrac{2u_{1}}{w^5}u_{11}^2+\dfrac{2u_1}{w^3(1+w)}\sum_{j}u_{1j}^2\right)+\sum_{i}\dfrac{\partial Q_k}{\partial A_{ii}}\dfrac{2u_1}{w^2(1+w)}u_{1i}^2
\\
&\hspace{.5cm}+2\sum_{i}\dfrac{\partial Q_k}{\partial A_{1i}}\left(\dfrac{u_1}{w^4}u_{11}u_{1i}+\dfrac{u_1}{w^2(1+w)}u_{1i}u_{ii}+\dfrac{u_1}{w^3(w+1)}u_{11}u_{i1}\right). 
\end{align*}
Then, by using the same bounds given in Equation \eqref{C flow}, it follows that
\begin{align*}
C\geq\dfrac{\partial Q_k}{\partial A_{11}}\left(\dfrac{2u_{1}^2}{w^2}-\left(1+\dfrac{2}{\ln u_1}\right)\right)\dfrac{u_{11}^2}{w^3u_1^2\ln u_1}
+\dfrac{2}{w^2(1+w)\ln u_1}\sum_{i}\dfrac{\partial Q_k}{\partial A_{1i}}u_{1i}u_{ii}.
\end{align*}
Now, from Equation \eqref{3.5.}, we have
\begin{align*}
\dfrac{2}{w^2(1+w)\ln u_1}\sum_{i}\dfrac{\partial Q_k}{\partial a_{1i}}u_{1i}u_{ii}&\geq-\dfrac{2u_1^2\ln u_1}{w^3(1+w)}\sum_{i>1}\dfrac{\rho_i^2}{\rho^2}\sum_{j\geq 1}\dfrac{\partial Q_k}{\partial a_{jj}}
\geq -\dfrac{8u_1^2\ln u_1}{w^3(1+w)\rho}\sum_{j\geq 1}\dfrac{\partial Q_k}{\partial a_{jj}},
\end{align*}
and 
\begin{align*}
\dfrac{\partial Q_k}{\partial a_{11}}\left(\dfrac{2u_{1}^2}{w^2}-\left(1+\dfrac{2}{\ln u_1}\right)\right)\dfrac{u_{11}^2}{w^3u_1^2\ln u_1}&\geq \dfrac{\partial Q_k}{\partial a_{11}}\left(\dfrac{2u_{1}^2}{w^2}-\left(1+\dfrac{2}{\ln u_1}\right)\right)\dfrac{u_1^2\ln u_1}{w^3}\dfrac{\varphi'^2}{4\varphi^2}
\\
&\geq\dfrac{\partial Q_k}{\partial a_{11}}\left(\dfrac{2u_{1}^2}{w^2}-\left(1+\dfrac{2}{\ln u_1}\right)\right)\dfrac{u_1^2\ln u_1}{w^3}\dfrac{1}{16M^2}
\\
&\geq\dfrac{C_2(k,n)}{64M^2}\dfrac{u_1^2\ln u_1}{w^3}\sum_{i\geq 1}\dfrac{\partial Q_k}{\partial a_{ii}},
\end{align*}
where $C_2(k,n)$ is the constant in \eqref{4.1.16}. Hence, it follows that
\begin{align*}
C\geq \sum_{i\geq 1}\dfrac{\partial Q_k}{\partial a_{ii}}\left(\dfrac{C_2(k,n)}{64M^2}\dfrac{u_1^2\ln u_1}{w^3}-\dfrac{8u_1^2\ln u_1}{w^3(w+1)\rho}\right).
\end{align*}
Then, by adding the bounds from the estimates from $B$ and $C$,  we obtain
\begin{align*}
0&\geq B+A
\\
&\geq  \sum_{i\geq 1}\dfrac{\partial Q_k}{\partial a_{ii}}\left(\dfrac{C_2(k,n)}{64M^2}\dfrac{u_1^2\ln u_1}{w^3}-8\dfrac{u_1^2\ln u_1}{w^3(w+1)\rho}-\dfrac{2}{w\rho} -\dfrac{u_1}{2M^2w^3}\right),
\end{align*} 
or equivalently,
\begin{align*}
\dfrac{C_2}{64M^2}\ln u_1\rho\leq 8\dfrac{\ln u_1}{(w+1)}+\dfrac{2w^2}{u_1^2}+\dfrac{\rho}{2M^2u_1}, 
\end{align*}
which leads to 
\begin{align*}
\rho\ln u_1\leq C(k,n)\left(M^2+Mr\right).
\end{align*}
Finally,
\begin{align*}
\ln|Du(0)|&\leq\dfrac{\varphi(u(x_0))\rho(x_0)\ln u_1(x_0)}{\varphi(u(0))\rho(0)}\leq C(k,n)\left(\dfrac{M^2}{r^2}+\dfrac{M}{r}\right).
\end{align*} 
\end{proof}

As a consequence of Theorem \ref{T 34} we obtain a non-existence result for graphical $Q_k$-translator. 
\begin{theorem}\label{T 36}
There are no solutions $u\in\mathcal{C}^3(\rr^n)$ to Equation \eqref{4.1.} such that 
\begin{enumerate}
	\item The principal curvatures of the graph of $u$ lies in $\Gamma_{k+1}$.
	\item  $u(x)=o(|x|)\mbox{ as }|x|\to\infty$.
\end{enumerate} 
\end{theorem}

\begin{proof}
Let $u$ be such solution to Equation (\ref{4.1}). Then, by property (2), it follows
\begin{align*}
\max\limits_{B_r}|u|\leq Cr,\:\forall r>1.
\end{align*}
Furthermore, by Theorem \ref{T 34}, we have that 
\begin{align*}
|Du(x)|\leq C_1,\mbox{ for all }x\in\rr^n. 
\end{align*}
We claim that $|Du|=0$ in $\rr^n$ and we argue this by contradiction, we assume that there exists some $\delta>0 $ such that 
\begin{align}\label{0.0}
	|Du(0)|\geq\delta.
\end{align}	 
Let $r>1$ and the test function given by
\begin{align*}
G(x)=\rho(x)g(u)|\nabla u|,
\end{align*} 
where $\rho(x)=r^2-|x|^2$, $g(u)=\left(1-\dfrac{u}{M}\right)^\beta$ and $M=\max\limits_{B_r}u$. The constant $\beta<0$ is still to be fixed. 
\newline
Note that $G:\overline{B_r}\to\rr$ attains its maximum at an interior point $x_0$. We also choose a coordinate system such that $u_{ij}(x_0)$ is a diagonal matrix for $2\leq i,j\leq n$, 
\begin{align*}
u_1(x_0)=|\nabla u(x_0)|\mbox{ and }u_i(x_0)=0,i\geq 2.  
\end{align*}
Let $\delta_1>0$ such that 
\begin{align}\label{3.0}
\rho(x_0)\geq \delta_1r^2\mbox{ and }u_1(x_0)\geq \delta_1.
\end{align}

We have the following equations at $x_0$,
\begin{align}\label{3.1}
0&=(\ln G)_i=\dfrac{\rho_i}{\rho}+\dfrac{g'}{g}u_i+\dfrac{u_{1i}}{u_1},
\\
\label{3.2}
0&\geq Q_{k,ij}(\ln G)_{ij}
\\ \notag
&=Q_{k,ij}\left(\dfrac{\rho_{ij}}{\rho}-\dfrac{\rho_i\rho_j}{\rho^2}+\dfrac{g'}{g}u_{ij}+\left(\dfrac{g''}{g}-\left(\dfrac{g'}{g}\right)^2\right)u_iu_j+\dfrac{u_{1ij}}{u_1}-\dfrac{u_{1i}u_{1j}}{u_1^2}\right).
\end{align}
Then, by Equation \eqref{3.1}, it follows that
\begin{align*}
\dfrac{u_{11}}{u_{1}}=-\dfrac{\rho_1}{\rho}-\dfrac{g'}{g}u_1\mbox{ and }\dfrac{u_{1i}}{u_{1}}=-\dfrac{\rho_i}{\rho}.
\end{align*}
Consequently, since $\beta<0$, we may enlarge $r$ such that $u_{11}(x_0)\leq 0$, which allows us to use the equations from Lemma \ref{L 3.1}. In addition, we also have
\begin{align*}
\dfrac{u_{1i}u_{1j}}{u_1^2}\leq 2\dfrac{\rho_i\rho_j}{\rho^2}+2\left(\dfrac{g'}{g}\right)^2u_iu_j.
\end{align*}
Then, it follows that
\begin{align}\label{3.3}
0\geq Q_{k,ij}\left(\dfrac{\rho_{ij}}{\rho}-3\dfrac{\rho_i\rho_j}{\rho^2}+\dfrac{g'}{g}u_{ij}+\left(\dfrac{g''}{g}-3\dfrac{g'^2}{g^2}\right)u_iu_j\right) +Q_{k,ij}\dfrac{u_{1ij}}{u_1}.
\end{align}
First, we estimate the first two terms in \eqref{3.3},
\begin{align*}
Q_{k,ij}\left(\dfrac{\rho_{ij}}{\rho}-3\dfrac{\rho_i\rho_j}{\rho^2} \right)&=-2\dfrac{Q_{k,ii}}{\rho}-12Q_{k,ij}\dfrac{x_ix_j}{\rho^2}\geq -2\dfrac{Q_{k,ii}}{\rho}-12\dfrac{Q_{k,ii}}{\rho^2}|x|^2
\\
&\geq-2\left( \dfrac{1}{r^2\delta_1}+\dfrac{6}{\delta_1^4r^2}\right)\dfrac{1}{w}\dfrac{\partial Q_k}{\partial A_{ii}},
\end{align*}
in the last line we  use Equation (\ref{3.0}). 
\\
Secondly, we estimate the third and fourth term from (\ref{3.3}). Note that at $x_0$ the following equations holds 
\begin{align*}
Q_{k,ij}u_{ij}=Q_k,\:\dfrac{g'}{g}=-\dfrac{\beta}{M}\left(1-\dfrac{u}{M}\right)^{-1}\mbox{ and }\dfrac{g''}{g}=\dfrac{\beta(\beta-1)}{M^2}\left(1-\dfrac{u}{M}\right)^{-2}.
\end{align*}
Therefore, we can drop the term $Q_{k,ij}\dfrac{g'}{g}u_{ij}>0$ from \eqref{3.3}. 
\newline
Furthermore, at $x_0$, we estimate
\begin{align*}
\dfrac{g''}{g}-3\left(\frac{g'}{g}\right)^2
\geq \dfrac{-2\beta^2 -\beta}{M^2}\geq \dfrac{3}{32M^2},
\end{align*}
which hold for $\beta\in\left(-\dfrac{1}{4},-\dfrac{1}{8}\right)$. Then, by combining the above estimates, it yields
\begin{align*}
Q_{k,ij}\left(\dfrac{g'}{g}u_{ij}+\left(\dfrac{g''}{g}-3\dfrac{g'^2}{g^2}\right)u_iu_j\right)\geq \dfrac{-2\beta^2+\beta}{M^2}Q_{k,11}u_1^2\geq \dfrac{C(k,n)}{M^2}\dfrac{u_1^2}{w^3}\dfrac{\partial Q_k}{\partial A_{ii}},
\end{align*}
in the last inequality we use Equation \eqref{4.1.16}.

Finally, for the last term in \eqref{3.3}, we can use Equation \eqref{3.5.} to show that
\begin{align*}
\frac{Q_{k,ij}u_{ij1}}{u_1}&\geq-\frac{2u_1}{w^3(1+w)}\frac{\partial Q_k}{\partial A_{ii}}\sum_{j>1}\frac{u_{1j}^2}{u_1^2}=-\frac{2u_1}{w^3(1+w)}\frac{\partial Q_k}{\partial A_{ii}}\sum_{j>1}\frac{\rho_j^2}{\rho^2}
\\
&\geq -\dfrac{8u_1}{w^4}\dfrac{1}{\delta_1r^2}\dfrac{\partial Q_k}{\partial A_{ii}}. 
\end{align*} 
Consequently, we can estimate  \eqref{3.3} as follows
\begin{align*}
\dfrac{8u_1}{w^4}\dfrac{1}{\delta_1r^2}\dfrac{\partial Q_k}{\partial A_{ii}}+2\left( \dfrac{1}{r^2\delta_1}+\dfrac{6}{\delta_1^2r^2}\right)\dfrac{1}{w}\dfrac{\partial Q_k}{\partial A_{ii}}\geq  \dfrac{C(k,n)}{M^2}\dfrac{u_1^2}{w^3}\dfrac{\partial Q_k}{\partial A_{ii}}.
\end{align*}
In particular, we have that $u_1(x_0)^2\leq C(k,n,\delta_1)\frac{M^2}{r^2}$, which implies that $u_1(0)\to 0$ as $r\to\infty$, a contradiction with Equation \eqref{0.0}. Therefore, $u$ is constant but this fact contradicts property (1) from the statement of Theorem \ref{T 36}. 
\end{proof}

\section{Second order Estimates}\label{sec4}
In this section we derive interior second order estimates for solutions to the $Q_k$-flow and $Q_k$-translator equation. For this purpose, we derive local uniform estimates for $H^2$ in both settings for cases $k\geq 1$. Recall that if $M_0$ satisfies $\lambda\in\Gamma_{k+1}$ for $k\geq 1$, then $|A|^2\leq H^2$ holds on $M_0$.

In the following we will assume that $M_0$ is locally the graph of a function over a hyperplane orthogonal to $w\in \Sp^n$.
\begin{lemma}\label{L 4.1}
	Let $M_t$  be a solution to the $Q_k$-flow (\ref{I.2}). Then, we have the following equations at $p\in M_t$:
	\begin{align}
		\label{4.e}
		\left(\partial_t-\square_k\right)u&=0, \mbox{ where }u=\pI{p, w},
		\\
		\label{4.a}
		\left(\partial_t-\square_k\right)h_{ii}&=\dfrac{\partial^2 Q_k}{\partial h_{cd}\partial h_{ab}}\nabla_ih_{cd}\nabla_ih_{ab}+|\A|_k^2h_{ij}-Q_kh_{il}h_j^l,
		\\ \label{4.b}
		\left(\partial_t-\square_k\right)\pI{\nu,w}&=|\A|^2_k\pI{\nu,w},
		\\ \label{56}
		\left(\partial_t-\square_k\right)v&=-v|\A|_k^2-2\norm{\nabla v}_k^2v^{-1},\mbox{ where } v=\pI{\nu,w}^{-1},
		\\ \label{57}
		\left(\partial_t-\square_k\right)H&=\dfrac{\partial^2 Q_k}{\partial h^d_{c}\partial h^b_{a}}\nabla^i h^d_{c}\nabla_i h^b_{a}+|\A|_k^2H,
		\\  \label{58}
		\left(\partial_t-\square_k\right)Q_k&=|\A|_k^2Q_k
		\\ \label{59}
		\left(\partial_t-\square_k\right)r^2&\leq 0,\mbox{ where }r^2=|p|^2+2\dfrac{n-k}{k+1}t,
		\end{align}
where $Q_{k,ij}=\dfrac{\partial Q_k}{\partial h_{ij}}$, $\square_kf=\sum\limits_{i,j}Q_{k,ij}\nabla_i\nabla_jf$,$\norm{X}_k^2=\pI{X,X}_k$, $\pI{X,Y}_k=\sum_{i,j}Q_{k,ij}X^iY^j$ and  $|\A|^2_k=\sum\limits_{i,j,l}Q_{k,ij}h_{il}h_{lj}$.  
\end{lemma}

\begin{proof}
 The proofs of equations \eqref{4.e}-\eqref{56} can be found in \cite{Choi-Daskalopoulos_2016}, and for equations \eqref{57} , \eqref{58} can be found in \cite{dieter_2005}. 
 \newline
  For Equation \eqref{59}, we use a normal frame at $p\in M_t$ given by $e_i\in T_pM_t$. Then, it follows 
	\begin{align*}
		\partial_tr^2&=2\pI{Q_k\nu,p}+2\dfrac{n-k}{k+1},
		\\
		\square_kr&=2Q_{k,ij}\delta^i_j+2\pI{Q_k\nu,p}.
	\end{align*}
	Therefore, by Equation \eqref{2.5.1} together with Corollary \ref{C 2.7}, we have that $(\partial_t-\square_k)r\leq 0$ since $Q_{k,ij}\delta^i_j\geq \frac{n-k}{k+1}$.    
\end{proof}

\begin{remark}
	Note that equation \eqref{56} implies that the assumption imposed  to $M_0$ is still valid under the $Q_k$-flow if the domain of $v$ lies in the support of $R^2-r^2$, where $R$ depends on the domain  of definition of $u$. 
\end{remark}

In \cite{Choi-Daskalopoulos_2016} Theorem 2.4, the authors give a  gradient estimate for the $Q_k$-flow of the form 
\begin{align*}
	\phi v \leq\sup_{M_0}\phi(p,0) v(p,0),
\end{align*}	
for a convex $M_0$ initial data such that $Q_k>0$. We give a similar estimate with the condition $\lambda\in\Gamma_{k+1}$ instead. 
\begin{proposition}
	Let $M_t$ be a solution to \eqref{I.2} such that $\lambda\in \Gamma_{k+1}$. Then, for  $R>0$ and $x_0\in\rr^{n+1}$, the following estimate holds
	\begin{align*}
		\phi_+(x,t)v(x,t)\leq \sup_{M_0}\phi_+(x,0)v(x,0),
	\end{align*}	 
	where $\phi_+=\max\{\phi,0\}$, 
	$\phi(x_0,t_0)=R^2-|x-x_0|^2-2\frac{n-k}{k+1}t$ and $v$ is defined in the support of $\phi_+$. 
\end{proposition}

\begin{proof}
	First, we note that without loosing generality its enough to consider $x_0=0$.  Moreover, let  
	\begin{align*}
		\eta(r)=(R^2-r^2)^2,
	\end{align*}
where $r^2=|x|^2+2\frac{n-k}{k+1}t$ as in Equation \eqref{59}. Then, it follos that
\begin{align}\label{4.3}
	\pare{\partial_t-\square_k}\eta v^2=v^2\pare{\partial_t-\square_k}\eta+2v\eta\pare{\partial_t-\square_k}v-2\eta\norm{\nabla v}^2_k-4v\pI{\nabla \eta,\nabla v}_k.
\end{align}
On one hand, we have 
\begin{align*}
(\partial_t-\square_k)\eta=2(R^2-r^2)\pare{\partial_t-\square_k}r^2-2\norm{\nabla r^2}_k^2\leq -2\norm{\nabla |x|^2}_k^2. 	
\end{align*}
Then, by apllying this equation together with Equation \eqref{56} onto Equation \eqref{4.3}, it follows that
\begin{align}\label{4.31}
	\pare{\partial_t-\square_k}\eta v^2\leq -2v^2\norm{\nabla |x|^2}_k^2-6\eta\norm{\nabla v}_k^2-2|\A|_k^2v^2 -4v\pI{\nabla \eta,\nabla v}_k.
\end{align}	
On the other hand, we have the term
\begin{align*}
	-4v\pI{\nabla \eta,\nabla v}_k=-6v\pI{\nabla\eta,\nabla v}_k+\eta^{-1}\pI{\nabla \eta,\nabla (\eta v^2)}_k-\eta^{-1}v^2\norm{\nabla\eta}_k^2,
\end{align*}	
then it holds
\begin{align}\label{4.32}
\pare{\partial_t-\square_k}\eta v^2\leq -6v^2\norm{\nabla |x|^2}_k^2-6\pI{\nabla\eta,\nabla v}_k-6\eta\norm{\nabla v}_k^2. 
\end{align}		 	
Finally, for the term
\begin{align*}
	-6v\pI{\nabla\eta,\nabla v}_k&=-6v\eta'Q_{k,ij}\nabla_i|x|^2\nabla_jv
	\\
	&\leq 6Q_{k,ij}\pare{(\eta' )^2\eta^{-1}\eps\dfrac{\nabla_i |x|^2\nabla_j |x|^2}{2}v^2+\eta\dfrac{\nabla_iv\nabla_jv}{2\eps}}
	\\
	&=6Q_{k,ij}\pare{\nabla_i |x|^2\nabla_j |x|^2v^2+\eta\nabla_iv\nabla_jv},	
\end{align*}	
where in the second line we use Young Inequality and for the last line we use $\eps=\frac{1}{2}$. Applying this onto \eqref{4.32} it follows that $\pare{\partial_t-\square_k}\eta v^2\leq 0$. Therefore, the estimate in Proposition \ref{4.3} holds since chaining $\eta$ by $\phi_+$ does not modify the estimate as long  $v$ is defined in the support of $\phi_+$.
\end{proof}

In the following estimate we assume that there is positive function $h(x,t)$ which satisfies
\begin{align}\label{f}
	\pare{\partial_t-\square_k}h\leq C(k,n)\mbox{ and }\norm{\nabla h}_k^2\leq C(k,n)h,\mbox{ on }M_t. 
\end{align}	

\begin{theorem}\label{T 4.4}
	Let $R>0$ such that $M_t=\set{x\in M_t:h(x,t)\leq R^2}$ is graph over a ball of radio $R$  in $[0,T]$. Then, for any $t_0\in[0,T] $ and $\theta\in [0,1]$, the estimate
	\begin{align*}
		\sup\limits_{M_t} H^2\leq \dfrac{c(k,n)}{(1-\theta)^2}\left(\frac{1}{t}+\frac{1}{R^2}\right)\sup\limits_{[0,t_0]}\sup\limits_{M_t}v^4.
	\end{align*}	
\end{theorem}	
The proof is very similar to the given in \cite{ecker_huisken_1991} for the Mean Curvature Flow. For completeness we give all the details here.

\begin{proof}
	Let $\varphi$ a real function to be chosen, and we consider 
	\begin{align*}
		\pare{\partial_t-\square_k}\pare{H^2\varphi(v^2)}=&2H\varphi\pare{\partial_t-\square_k}H-2\varphi\norm{\nabla H}_k^2+2vH^2\varphi'\pare{\partial_t-\square_k}v
		\\
		&-2H^2\varphi'\norm{\nabla v}_k^2-2\pI{\nabla H^2,\nabla\varphi}_k-4H\varphi''v^2\norm{\nabla v}_k^2.
	\end{align*}
Then, by replacing the above equation with equations \eqref{56} and \eqref{57} together with the concavity of $Q_k$, it follows
\begin{align}\label{4.4}
	\pare{\partial_t-\square_k}H^2\varphi(v^2)\leq&  2H^2\varphi|\A|_k^2-2\varphi\norm{\nabla H}_k^2-2v^2H^2\varphi'|\A|_k^2
	\\ \notag
	&-6H^2\varphi'\norm{\nabla v}_k^2-2\pI{\nabla H^2,\nabla\varphi}_k-4H\varphi''v^2\norm{\nabla v}_k^2.
\end{align}
On one hand, we have for the term
\begin{align*}
-2\pI{\nabla H^2,\nabla\varphi}_k&=-\varphi^{-1}\pI{\nabla(H^2\varphi),\nabla\varphi}_k+H^2\varphi^{-1}\norm{\nabla \varphi}_k^2-4H\varphi'v\pI{\nabla H,\nabla v }_k
\\
&\leq-\varphi^{-1}\pI{\nabla(H^2\varphi),\nabla\varphi}_k+6H^2\varphi^{-1}(\varphi')^2v^2\norm{\nabla v}_k^2+2\varphi\norm{\nabla H}_k^2.	
\end{align*}	
In the last line we use Young Inequality with the last term in the right side of the first line. Then, by substituting the last equation onto  \eqref{4.4}, it follows 
\begin{align}
	\pare{\partial_t-\square_k}H^2\varphi(v^2)\leq &2|\A|_k^2H^2(\varphi-v\varphi')-\varphi^{-1}\pI{\nabla H^2\varphi,\nabla \varphi}_k
	\\ \notag
	&-\norm{\nabla v}_k^2H^2\pare{6\varphi'(1-\varphi'\varphi^{-1} v^2)+4\varphi''v^2}.
\end{align}	
Now we let $g=H^2\varphi(v^2)$ and we choose $\varphi(x)=\frac{x}{1-ax}$, where $a$ is still to be chosen. In addition, since $|A|_k^2\leq (n-k)H^2$, it follows
\begin{align}\label{4.41}
	\pare{\partial_t-\square_k}g&\leq 2(n-k)\varphi^{-2}(\varphi-v^2\varphi')g^2-\varphi^{-1}\pI{\nabla g,\nabla\varphi}_k
	\\ \notag
	&-g\norm{\nabla v}_k^2\pare{6\varphi'(1-\varphi\varphi'v^2)+4\varphi''v^2},
\end{align}	
and 
\begin{align*}
	\begin{cases}
		&\varphi-v^2\varphi'=-a\varphi^2,
		\\		
		&\varphi^{-1}\pI{\nabla g,\nabla\varphi}_k=2\varphi v^{-3}\nabla v,
		\\
		&\pare{6\varphi'(1-\varphi'\varphi^{-1} v^2)+4\varphi''v^2}=\dfrac{2a}{(1-ax)^2}.
	\end{cases}	
\end{align*}	
Consequently, by substituting these equations onto Equation \eqref{4.41}, we have
\begin{align}\label{4.42}
	\pare{\partial_t-\square_k}g&\leq -2(n-k)ag^2-2\varphi v^{-3}\pI{\nabla g,\nabla\varphi}_k-\dfrac{2ag}{(1-av^2)^2}\norm{\nabla v}_k^2.
\end{align}	
On the other hand, we consider $\eta(x,t)=(R^2-f(x,t))^2$. Then, the evolution equation for $\eta$  satisfies
\begin{align}\label{4.43}
	\pare{\partial_t-\square_k}\eta\leq 2C(k,n)R^2-2\norm{\nabla f}_k^2,
\end{align}	
 recall that we are using Equation \eqref{f}. 
 \newline
Furthermore, it follows that 
 \begin{align}\label{4.44}
 	\pare{\partial_t-\square_k}g\eta=\eta\pare{\partial_t-\square_k}g+g\pare{\partial_t-\square_k}\eta-2\pI{\nabla\eta,\nabla g}_k.
 \end{align}	
 Note that the term
 \begin{align*}
 	-2\pI{\nabla\eta,\nabla g}_k=-2\eta^{-1}\pI{\nabla \eta,\nabla g\eta}_k+8g\norm{\nabla f}_k^2.
 \end{align*}
 Then, by substituting equations \eqref{4.42}, \eqref{4.43} with the this last equation onto Equation \eqref{4.44}, it follows
 \begin{align} \label{4.45}
 	&\pare{\partial_t-\square_k}g\eta
 	\\ \notag
 	&\leq -2a(n-k)g^2\eta-\dfrac{2ag\eta}{(1-av^2)^2}\norm{\nabla v}_k^2+2C(k,n)gR^2
 	\\ \notag
 	&-2\eta^{-1}\pI{\nabla g\eta,\nabla \eta }+6g\norm{\nabla f}_k^2-2\varphi\eta v^{-3}\pI{\nabla g,\nabla v}_k.
 \end{align}	 
 Moreover, for the term
 \begin{align*}
 	-2\varphi\eta v^{-3}\pI{\nabla g,\nabla v}_k&=-2v^{-3}\varphi\pI{\nabla g\eta,\nabla v }_k+4g\pI{\eta^{1/2}\nabla v,v^{-3}\varphi\nabla r}_k
 	\\
 	&\leq-2v^{-3}\varphi\pI{\nabla g\eta,\nabla v }_k+ g\pare{2\eta\dfrac{2\norm{\nabla v}_k^2 }{\eps}+2\varphi^2 v^{-6}\eps\norm{\nabla f}_k^2}
 	\\
 	&=-2v^{-3}\varphi\pI{\nabla v,\nabla g\eta}_k+2g\eta\dfrac{a}{(1-av^2)^2}\norm{\nabla v}_k^2+\dfrac{g}{v^2a}\norm{\nabla f}_k ^2.
 \end{align*}	
 Note that in second line we use Young Inequality, in third line we use $\eps=\frac{(1-av^2)^2}{a}$	and  $\varphi^2v^{-6}\eps=v^{-2a}$. Then, by applying this equation onto  Equation \eqref{4.45}, it follows that
 \begin{align}\label{4.46}
 \pare{\partial_t-\square_k}g\eta \leq& 	-2a(n-k)\eta g^2+C(k,n)g\pare{ f\pare{\frac{1}{v^2a}+1}+R^2}
 \\ \notag
 &-2\pI{\nabla g\eta, v^{-3}\varphi\nabla v-2\eta^{-1}\nabla \eta }_k.
 \end{align} 
 Finally, we consider the test function $G=t\eta g$. Note that $G$ reaches it maximum at $t_0>0$ and in a interior point of $M_{t_0}$. Therefore,
 \begin{align}\label{4.47}
  0\leq \pare{\partial_t-\square_k}G\leq -2a g^2\eta t_0+C(k,n)gt_0\pare{ \frac{f}{v^2a}+f+R^2}+g\eta. 
 \end{align}
 Then, by multiplying  $\frac{\eta t_0}{2a(n-k)}$ to Equation \eqref{4.47}, it follows that 
 \begin{align*}
 	m(T)^2\leq m(T)\frac{C(k,n)}{2a(n-k)}\pare{t_0\pare{\frac{f}{v^2a}+f+R^2} +\eta},
 \end{align*}	
 where $m(T)=\sup_{[0,t_0]}\sup_{M_t} G$.  Consequently, it follows that 
 \begin{align*}
 	m(T)\leq C(k,n)v^2\pare{tR^2+R^4  }, 
 \end{align*}
 provided that $a>\frac12\inf_{[0,T]}\inf_{M_t}v^{-2}$. 
 \newline
 Therefore, since $\varphi(v^2)\geq v^{-2}$ and $\eta>(1-\theta)^2R^4$ in $\set{x\in M_t:f(x,t)\leq\theta R^2}$ for $\theta\in[0,1)$, the estimate from Theorem \ref{T 4.4} follows since 
 \begin{align*}
 	H^2\leq v^2\frac{m(T)}{R^4t(1-\theta)^2},\mbox{ holds on} M_t.
 \end{align*}	 
\end{proof}

Now we develop a similar estimate for Equation (\ref{4.1.}). Recall that we are assuming that $M_0=M$ is locally a graph over a hyperplane orthogonal to $w\in\Sp^n$.

\begin{lemma}\label{4.2..}
	Let $M$ be a  $Q_k$-translator. Then, we have the following equations at $p\in M$:
	\begin{align}
		\label{4.4.3}
		&\square_ku=Q_k\pI{\nu,w}, \mbox{ where } u=\pI{p,w},
		\\
		&\square_kr^2=2\pare{Q_k\pI{\nu,p} +Q_{k,ij}\delta^i_j}, \mbox{ where } r^2=|p|^2,
		\\\label{4.4.6}
		&\square_k H+\sum_{i=1}^nQ_{k,ab;cd}\nabla^ih_a^b\nabla_ih_c^d+|\A|_k^2H+\pI{\nabla H,\eps_{n+1}}=0,
		\\ \label{Q_k}
		&\square_k Q_k+|\A|_k^2Q_k+\pI{\nabla Q_k,\eps_{n+1}}=0,
		\\ \label{4.2.1.}
		&\square_k v-\pI{\nabla v,\eps_{n+1}}+|\A|^2_kv-2v^{-1}\norm{\nabla v}_k^2=0,\mbox{ where }v=\pI{\nu, w}^{-1}.
	\end{align}
\end{lemma}

\begin{proof}
For the coefficient of the Second Fundamental Form of $M$ we have
	\begin{align*}
		\nabla_j\nabla_iQ_k=&\dfrac{\partial^2 Q_k}{\partial h_{cd}\partial h_{ab}}\nabla_jh_{cd}\nabla_ih_{ab}+\square_kh_{ij}+|\A|_k^2h_{ij}
	\\
	&-Q_kh_{il}h_{lj}+\dfrac{\partial Q_k}{\partial h_{ab}}\left(h_{ib}h_{am}h_{mj}-h_{im}h_{aj}h_{mb}\right).
	\end{align*}
	Here we use the equations from Theorem 2.1 given in \cite{huisken_polden_1996}.
	\newline
	Moreover, by translating soliton Equation (\ref{I.3}), it follows that
	\begin{align*}
		\nabla_j\nabla_iQ_k&=-\pI{\nabla h_{ij},\eps_{n+1}}-h_{il}h_{jl}Q_k.
	\end{align*}
  Then, after combining both equations  for $\nabla_j\nabla_i Q_k$, we obtain
	\begin{align*}
		\square_kh_{ij}=& -\dfrac{\partial^2 Q_k}{\partial h_{cd}\partial h_{ab}}\nabla_jh_{cd}\nabla_ih_{ab}
		\\
		&-|\A|^2_kh_{ij}-\pI{\nabla h_{ij},\eps_{n+1}}-\dfrac{\partial Q_k}{\partial h_{ab}}\left(h_{ib}h_{am}h_{mj}-h_{im}h_{aj}h_{mb}\right).
	\end{align*}
  Consequently, Equation \eqref{4.4.6} follows by taking trace in the above equation.
\newline
For the function $\pI{\nu,w}$, it follows that 
\begin{align*}
\nabla_i\pI{\nu,w}=-h_{il}\pI{e_l,w}\mbox{ and }\nabla_j\nabla_i\pI{\nu,w}=-\nabla_lh_{ij}\pI{e_l,w}-h_{il}h_{lj}\pI{\nu,w},
\end{align*}
note that we also use the Codazzi Equations in the last line. Therefore, it holds 
\begin{align}\label{nu,w}
	\square_k\pI{\nu,w}+|\A|_k^2\pI{v,w}+\pI{\nabla\pI{\nu,w},\eps_{n+1}}=0.
\end{align}
Finally, for  $v$, we have
\begin{align*}
	\nabla v=-v^{2}\pI{\nu,w},\mbox{ and }\square_kv=-v^2\square_k\pI{\nu,w}-2v^{-1}\norm{\nabla v}_k^{2}.
\end{align*}
Then, by substituting the above equations with Equation \eqref{nu,w}, Equation \eqref{4.2.1.} holds. Note that Equation equations \eqref{Q_k}  follows by taking $w=\eps_{n+1}$.
\end{proof}

Now we derive an Ecker Huisken interior estimate for $Q_k$-translators with principal curvature vector $\lambda\in \Gamma_{k+1}$. As in the parabolic case, we assume that there is a positive function $h(x)$ which satisfies
\begin{align}\label{f1}
	|\square_kh|\leq C(k,n)\mbox{ and }\norm{\nabla h}_k^2,|\nabla h|^2\leq C(k,n)h,\mbox{ on }M. 
\end{align}	

\begin{theorem}\label{T 4.7}
	Let $R>0$ such that $\set{x\in M:h(x)\leq R^2}$ is a ball of radio $R>0$ in the hyperplane orthogonal to $w$. Then, for any  $\theta\in [0,1]$, the estimate
	\begin{align}\label{H}
		 H^2\leq \dfrac{c(k,n)}{(1-\theta)^2}\left(1+\frac{1}{R^2}\right)\sup\limits_{\set{h\leq\theta R^2}}v^4.
	\end{align}	
\end{theorem}	
\begin{proof}
	The proof is very similar to the given in \ref{T 4.4}. Therefore, we only point put the main differences. We consider the test function 
	\begin{align*}
		G(x)=H^2\varphi(v^2)\eta(h),
	\end{align*}
where $\varphi(x)=\dfrac{x}{1-ax}$, $a=\frac{1}{2}\inf v^{-2}$, $\eta(x)=(R^2-h)^2$ and $h$ satisfies \eqref{f1}. 
\newline
First, we note that Equation \eqref{4.42}, is replaced by
\begin{align*}
	-\square_kg&\leq -2(n-k)ag^2-2\varphi v^{-3}\pI{\nabla g,\nabla\varphi}_k-\dfrac{2ag}{(1-av^2)^2}\norm{\nabla v}_k^2-\pI{\nabla g,\eps_{n+1}},
\end{align*}	
where $g=H^2\varphi$. Then, Equation \eqref{4.43}, is replaced by
\begin{align*}
-\square_k\eta\leq 2C(k,n)R^2-2\norm{\nabla h}_k^2.
\end{align*}
Consequently, by combining the above equations, Equation \eqref{4.47} is replaced by
\begin{align*}
-\square_kG\leq& \pI{\nabla G,-2\eta\nabla \eta-2\varphi\nabla }_k+\pI{\nabla G,\eps_{n+1}}+6g\norm{\nabla h}_k^2+\dfrac{g}{av^2}
\\
&+G\pare{ 2C(k,n)R^2-2C(k,n)ag-4|\nabla h|^2}.
\end{align*}	
Finally, since $G$ reaches it maximum at an interior point of $\set{h\leq\theta R^2}$, it follows that 
\begin{align*}
	m\leq C(k,n)\frac{R^4}{a}\left(1+\frac{1}{R^2}\pare{\frac{1}{av^2}+1}\right),
\end{align*}	
here we use that $\eta\leq R^4$, $h\leq R^2$, $av^2<2^{-1}$ and $m=\sup\limits_{\set{h\leq \theta R^2}}G$. Note that the estimates for $H^2$ follows since $\varphi\geq 1$ and $\eta\geq (1-\theta)^2R^4$. 
\end{proof}
Now we establish the curvature interior estimates for $Q_k$-translators.
\begin{corollary}\label{C 48}
Let $M$ be a $Q_k$-translator such that $\lambda\in\Gamma_{k+1}$ with $k\geq 1$. Then, for  $p\in M$, $|A|\leq C(k,n,R)$ in $B(p,R)$.
\end{corollary}

\begin{proof}
Let  $h^2=r^2-u^2$, where $r^2=|p|^2$ and $u=\pI{p,w}$.  Note that $h$ measures the distance of $M$ to the hyperplane orthogonal to $w$. In addition, we chose a normal frame $\set{e_i}$ of $T_pM $ at $p$. Then, by taking covariant derivatives on $u$, we have $\nabla_i u=\pI{e_i,w}$ and $\nabla_j\nabla_i u=h_{ij}\pI{\nu,w}$. Therefore,
	\begin{align*}
	\nabla u=w^{\top} \mbox{ and } \square_k u=Q_k\pI{\nu,w}.
	\end{align*}
Likewise, we have $\nabla_ir^2=2\pI{e_i,p}$ and $\nabla_j\nabla_i r^2=2\pare{h_{ij}\pI{\nu,p}+\delta_i^j}$. Consequently,
	\begin{align*}
 \square_k r^2=2\pare{Q_k\pI{\nu,p} +Q_{k,ij}\delta^i_j}.		
	\end{align*}	
Then, by combining the above equations for $h$, it follows that 
	\begin{align*}
		\nabla h=2(p-w)^{\top},\mbox{ and }\norm{\nabla h}_k^2,|\nabla h|^2\leq C(k,n)h. 
	\end{align*}	
And, for the second derivatives it follows that
	\begin{align*}
		\square_k h=2Q_k(p-uw)^{\perp}+2Q_{k,ij}\delta_{ij}-2\norm{w^{\top}}^2_k.
	\end{align*}
Consequently,  $\abs{\square_k h}\leq C(k,n)$.
	\newline
As a result, we may apply Theorem \ref{T 4.7} on $\set{h^2\leq R^2}$, to obtain
	\begin{align*}
		H^2\leq C(k,n,R)\sup_{B(p,R)}v^4.
	\end{align*}
Finally, by the gradient estimate from Theorem\ref{T 34}, we have that $v\leq C(k,n,R)$ at $B(p,R)$. Therefore, Corollary \ref{C 48} will follows by $|A|^2\leq H^2$.  
\end{proof}

\end{document}